\documentclass[A4paper, 12pt]{preprint}
\usepackage{amsfonts,amsmath,amssymb, times}
\usepackage{mhequ, amsthm}
\usepackage{color}
\newtheorem{theorem}{Theorem}[section]
\newtheorem*{theorem*}{Theorem}
\newtheorem*{corollary*}{Corollary}
\newtheorem{lemma}[theorem]{Lemma}
\newtheorem{proposition}[theorem]{Proposition}
\newtheorem{corollary}[theorem]{Corollary}

\numberwithin{equation}{section}

\begin{document}

\def\ssm{\smallsetminus}
\newcommand{\tmname}[1]{\textsc{#1}}
\newcommand{\tmop}[1]{\operatorname{#1}}
\newcommand{\tmsamp}[1]{\textsf{#1}}
\newenvironment{enumerateroman}{\begin{enumerate}[i.]}{\end{enumerate}}
\newenvironment{enumerateromancap}{\begin{enumerate}[I.]}{\end{enumerate}}

\newcounter{problemnr}
\setcounter{problemnr}{0}
\newenvironment{problem}{\medskip
  \refstepcounter{problemnr}\small{\bf\noindent Problem~\arabic{problemnr}\ }}{\normalsize}
\newenvironment{enumeratealphacap}{\begin{enumerate}[A.]}{\end{enumerate}}
\newcommand{\tmmathbf}[1]{\boldsymbol{#1}}

\def\paral{/\kern-0.55ex/}
\def\parals_#1{/\kern-0.55ex/_{\!#1}}
\def\bparals_#1{\breve{/\kern-0.55ex/_{\!#1}}}
\def\n#1{|\kern-0.24em|\kern-0.24em|#1|\kern-0.24em|\kern-0.24em|}

\newcommand{\A}{{\bf \mathcal A}}
\newcommand{\B}{{\bf \mathcal B}}
\def\C{\mathbb C}
\def\D{\mathcal D}
\newcommand{\dom}{{\mathcal D}om}
\newcommand{\pathR}{{\mathcal{\rm I\!R}}}
\newcommand{\Nabla}{{\bf \nabla}}
\newcommand{\E}{{\mathbf E}}
\newcommand{\Epsilon}{{\mathcal E}}
\newcommand{\F}{{\mathcal F}}
\newcommand{\G}{{\mathcal G}}
\def\g{{\mathfrak g}}
\newcommand{\HH}{{\mathcal H}}
\def\h{{\mathfrak h}}
\def\k{{\mathfrak k}}
\newcommand{\I}{{\mathcal I}}
\def\LL{{\mathbb L}}
\def\law{\mathop{\mathrm{ Law}}}
\def\m{{\mathfrak m}}
\newcommand{\K}{{\mathcal K}}
\newcommand{\p}{{\mathfrak p}}
\newcommand{\R}{{\mathbb R}}
\newcommand{\Rc}{{\mathcal R}}
\def\T{{\mathcal T}}
\def\M{{\mathcal M}}
\def\N{{\mathcal N}}
\newcommand{\pnabla}{{\nabla\!\!\!\!\!\!\nabla}}
\def\X{{\mathbb X}}
\def\Y{{\mathbb Y}}
\def\L{{\mathcal L}}
\def\1{{\mathbf 1}}
\def\half{{ \frac{1}{2} }}
\def\vol{{\mathop {\rm vol}}}
\def\euc{{\mathop {\rm eul}}}

\newcommand{\term}{{T}}
\newcommand{\tm}{{t}}
\newcommand{\stm}{{s}}
\newcommand{\pole}{{y_0}}

\def\ad{{\mathop {\rm ad}}}
\def\Conj{{\mathop {\rm Ad}}}
\def\Ad{{\mathop {\rm Ad}}}
\newcommand{\const}{\rm {const.}}
\newcommand{\eg}{\textit{e.g. }}
\newcommand{\as}{\textit{a.s. }}
\newcommand{\ie}{\textit{i.e. }}
\def\s.t.{\mathop {\rm s.t.}}
\def\esssup{\mathop{\rm ess\; sup}}
\def\Ric{\mathop{\rm Ric}}
\def\div{\mathop{\rm div}}
\def\ker{\mathop{\rm ker}}
\def\Hess{\mathop{\rm Hess}}
\def\Image{\mathop{\rm Image}}
\def\Dom{\mathop{\rm Dom}}
\def\id{\mathop {\rm id}}
\def\Image{\mathop{\rm Image}}
\def\Cyl{\mathop {\rm Cyl}}
\def\Conj{\mathop {\rm Conj}}
\def\Span{\mathop {\rm Span}}
\def\trace{\mathop{\rm trace}}
\def\ev{\mathop {\rm ev}}
\def\supp{{\mathrm supp}}
\def\Ent{\mathop {\rm Ent}}
\def\tr{\mathop {\rm tr}}
\def\graph{\mathop {\rm graph}}
\def\loc{\mathop{\rm loc}}
\def\so{{\mathfrak {so}}}
\def\su{{\mathfrak {su}}}
\def\u{{\mathfrak {u}}}
\def\o{{\mathfrak {o}}}
\def\pp{{\mathfrak p}}
\def\gl{{\mathfrak gl}}
\def\hol{{\mathfrak hol}}
\def\z{{\mathfrak z}}
\def\t{{\mathfrak t}}
\def\<{\langle}
\def\>{\rangle}
\def\span{{\mathop{\rm span}}}
\def\diam{\mathrm {diam}}
\def\inj{\mathrm {inj}}
\def\Lip{\mathrm {Lip}}
\def\Iso{\mathrm {Iso}}
\def\Osc{\mathop{\rm Osc}}
\renewcommand{\thefootnote}{}
\def\supp{\mathrm {Supp}}
\def\V{\mathbb V}
\def\vol{{\mathop {\rm vol}}}
\def\cut{{\mathop {\rm Cut}}}
\def\Lip{{\mathop {\rm Lip}}}

\def\paral{/\kern-0.55ex/}
\def\parals_#1{/\kern-0.55ex/_{\!#1}}
\def\bparals_#1{\breve{/\kern-0.55ex/_{\!#1}}}
\def\n#1{|\kern-0.24em|\kern-0.24em|#1|\kern-0.24em|\kern-0.24em|}
\def\f{\frac}
\title{First Order Feynman-Kac Formula}
\author{Xue-Mei Li  and
 James Thompson}
\institute{Department of Mathematics, Imperial College London,  U.K.\\
Math\'ematiques, Universit\'e du Luxembourg }
\date{}
\maketitle

\begin{abstract}
We study the parabolic integral kernel  for the weighted Laplacian with a potential. For manifolds with a pole we deduce formulas and estimates for the derivatives of the Feynman-Kac kernels and their logarithms,  these are in terms
of a `Gaussian' term and the semi-classical bridge.  \end{abstract}
\\[1em]

\small{AMS Mathematics Subject Classification : 60J60, 60H30, 58J65, 58J70}\\
{key words: manifold with a pole, semi-classical bridge,  Feynman-Kac formula, Gaussian bounds for fundamental solutions of parabolic equations, logarithmic heat kernels} \normalsize


\section{Introduction}
Throughout the article $\lbrace \Omega,\mathcal{F},\mathcal{F}_t,\mathbb{P}\rbrace$ is a  filtered probability space with right continuous and complete filtrations,   $M$ is a complete connected smooth Riemannian manifold of dimension $n$, $h$ is a smooth real valued function on $M$, and  $V$ is a real valued bounded H\"older continuous function on $M$.
We use $\Delta$ to denote the Laplace-Beltrami operator and use  $\Delta^h$ to denote  the weighted Laplace-Beltrami operator $\Delta^h=\Delta+2L_{\nabla h}$.

We study the parabolic equation 
$$\begin{aligned}
&\frac{\partial }{\partial t}f= (\f 12 \Delta^h -V) f, \quad    t>0,\\
& \lim_{t \downarrow 0} f(t,x) = f(x).\end{aligned}$$  
It is well know that  there exists a fundamental solution, which we refer as the Feynman-Kac kernel  and also as the heat kernel in case both $h$ and $V$ vanish identically and as the weighted heat kernel in case $V$  vanishes.

Our main objectives are to obtain precise Gaussian bounds, explicit formulas for the gradient of  the Feynman-Kac kernel and for the gradient of its logarithm, and small time asymptotics. We begin with developing suitable first order Feynman-Kac formulas, and make no assumptions on the uniform ellipticity of the operator. All assumptions will be on Riemannian data, but without differentiating the curvature.  

 Estimates and gradient estimates for the heat kernels is a classical topic, see for example  \cite{ Li-Yau, Davies-Gaussian-bound}. For the Schr\"odinger operators acting on real valued functions, it is very popular to use (zero order)  Feynman-Kac formulas, see \cite{Molchanov,  Azencott, Simon-82, Freidlin-book,  Sturm-93,  Demuth-vanCasteren,  Lorinczi-Hiroshima-Betz, Gueneysu}. See  \cite{Albverio-Kawabi-Rockner, Gross,Fitzsimmons} for its use in infinite dimensional analysis and for identifying  Dirichlet spaces, and see \cite{Elworthy-Truman-81} for its study in connection
with  Hamilton-Jacob equations. Differential formulas for   Feynman-Kac semigroups were obtained in \cite{Elworthy-Li, Elworthy-Li-Vilnius, Daprato-Zabczyk, DaPrato-2015, Li-Zhao}, for partial differential equations  and for stochastic partial differential equations. Heat kernel formula for Schr\"odinger type operator acting on sections of vector bundles can be found in
 \cite{Ndumu-09} and \cite{Braverman}.
 The study of the probability measure induced by the semi-classical bridge can be found  in \cite{Li-ibp-sc, Li-Hessian, Li-Hessian2}, c.f. \cite{Li-generalised} . See also \cite{Chen-Li-Wu-cut-off} for small time gradient and Hessian estimates on any complete Riemannian manifold with which we also obtain a Log-Sobolev inequality with a potential term and  a local Log-Sobolev inequality on the loop space.

The paper is organised as following, in Section 2 we obtain two basic formulas for the gradient of the Feynman-Kac semigroups on a general manifold, which leads to  gradient estimates for the semi-groups and for the fundamental solutions. The latter are given in terms of the ratio of fundamental solutions, and so for  explicit estimates we apply a Harnack inequality on the weighted heat kernel.
 In section 3 we work on manifolds with a pole and use the formulas mentioned earlier to obtain formulas for the fundamental solutions. The latter sets of formulas are explicit, versus implicit, from which we obtain estimates for their gradients in terms of the Gaussian kernels and on the Jacobian determinant of the exponential map.

\subsection{Notation} 
Without loss of generality we assume that $V\ge 0$ and  use  $P^{h,V}_tf$ to denote its solution, which is also written as $P_t^Vf$ if $h=0$,  as $P_t^hf$ if $V=0$, and as $P_tf$ if both $h$ and $V$ vanish.
Their corresponding integral kernels are denoted by the lower case functions: $p_t^{h,V}$, $p_t^V$, $p_t^h$, and $p_t$. Also, if $Z$ is an additional  $C^1$ vector field the notations $p_t^{h,Z,V}$ etc. will be used.

 We use $L_\infty$, $C^k$, $BC^k$, $C_K^\infty$, and $C^{1, \alpha}$ to denote respectively the set of essentially bounded measurable (real valued) functions, the set of  functions (or differential forms)  whose derivatives up to order $k$ exist and are continuous, the set of bounded $C^k$ functions with bounded derivatives, the set of $C^k$  functions with compact supports, and 
the set of functions whose first order derivative is H\"older continuous of order $\alpha$.

Let $x_0$ denote a point in $M$ and let $OM$ denote the space of orthonormal frames on $M$.
Let $u_t$ be the canonical horizontal Brownian motion on $OM$ whose projection to $M$ is a Brownian motion with drift $\nabla h$ (an $h$-Brownian motion) and will be denoted by $x_t$. We note that $x_t$ does not explode if and only if $u_t$ does not and $u_tu_0^{-1}$ is the stochastic parallel translation along  $x_t$, allowing covariant differentiating  vector fields along the non-differentiable paths of the Brownian motion.  Let $\Ric$ denote the Ricci curvature  and  let 
${\Ric}^{\sharp}_x:T_xM \rightarrow T_xM$ denote  the linear map defined by $\< {\Ric}^\sharp_x u,v\> = {\Ric_x}(u,v)$. Let  $W_t: T_{x_0}M\to T_{x_t}M$ denote the damped parallel translation solving the equation 
$$\f {d }{dt} \left(u_t^{-1}W_t\right)=-\f 12 u_t^{-1}\left( {\Ric}_{x_t}^{\sharp} -2\Hess(h)\right)({ W_t}), \quad W_0=Id.$$
Here $\Hess(h)=\nabla d h$ denote the Hessian of $h$.

\subsection{Main results for a general  manifold.}
Let $T$ be a positive number and let $0\le t\le T$.  We  set $$\rho^h(x)=\inf_{|v|=1, v\in T_xM} \{\Ric(v,v)-2 \Hess h(v,v)\}, \qquad \V_t = e^{-\int_0^t V(x_r) dr}.$$
If $\rho^h$ is bounded from below then the $h$-Brownian motion  does not explode. The following is  a key preliminary result. 

\begin{theorem*}
[ \ref{Feynman-Kac-kernel-1} \& \ref{thm:maintheorem}]
 Suppose that $\rho^h$ is bounded from below and $f\in L_\infty$.  Let $0< t \leq T$ and $v\in T_{x_0}M$. Then the following two formulas hold:
 \small
$$\begin{aligned}
(dP^{h,V}_Tf)(v) =\text{ }& \frac{1}{t}\E \left[ \V_T f(x_T) \left( \int_0^t \< W_s(v),u_s \,dB_s\>-\int_0^t \int_0^r  dV(W_s(v)) \,ds\,dr\right)\right],\\
(dP^{h,V}_Tf)(v) = &\frac{1}{t} \E \left[ f(x_t) \int_0^t \< W_s(v),u_s dB_s\>\right]\\
&+ \E \left[   f( x_t)\int_0^t \left( e^{-\int_{t-s}^t V(x_{r})\,dr}
\f {V (x_{t-s})} {t-s}\int_0^{t-s} \< W_r(v),u_r\, dB_r\>\right)\,ds\right].\end{aligned}$$
\end{theorem*}
 The first identity is for $V \in BC^1$, while the second identity is for H\"{o}lder continuous and bounded potentials, the first formula was given in \cite{Elworthy-Li} where  the method of differentiating a stochastic flow is used and hence involves a different set of conditions.

From these formulas, corresponding formulas for the  fundamental solutions follow, which would be in terms of the classical Brownian bridges (the conditioned Brownian motions terminating at $y_0$ at time $t$). Integration by parts formulas related to the above mentioned formulas, for the case of $V=0, h=0$ and for $M$ compact, was proved in \cite{Bismut}, however the conditioned Brownian motions are defined in terms of the logarithmic derivatives of the fundamental solution $p_t^h$,  and in this sense these probabilistic representations for the fundamental solutions are implicit. For manifolds with a pole, we will shortly introduce intrinsic versions of the formulas, in terms of the semi-classical bridges.

For a general manifold, gradient estimates can also be deduced directly from the identities mentioned earlier, for example from the first identity we obtain the following. 
Suppose that $\Ric - 2\Hess (h) \geq K$ and $V \in  BC^1$. Let $f$ be a non-negative  bounded measurable function with $P_t^{h,V}f(x_0)=1$. Then, by  Proposition~\ref{prop:derivativeestiamte},\begin{equation*}
| \nabla P^{h,V}_t f|_{x_0} \leq \frac{1}{\sqrt{t}}\Big(2C_1(t) \E\left[(f(x_t)\V_t\log (f(x_t)\V_t))^+\right] \Big)^{\frac{1}{2}}+tC_2(t)|\nabla V|_{\infty},
\end{equation*}
where the constants $C_1(t)$ and $tC_2(t)$ depend on $K$ and are of order $O(1)$ for $t$ small. If  furthermore $f$ is not identically zero  then,  
\begin{equation*}
| \nabla \log P^{h,V}_t f |_{x_0} \leq \frac{1}{\sqrt{t}}\Big(2C_1(t) {\mathcal H}_t(f,x_0) \Big)^{\frac{1}{2}}+tC_2(t)|\nabla V|_{\infty}, 
\end{equation*}
where ${\mathcal H}_t(f,x_0) $ is  the entropy function
$ \mathbb{E}\left[
\frac{f(x_t)\mathbb{V}_t}{P^{h,V}_tf(x_0)} \left( \log\left( \frac{f(x_t)\mathbb{V}_t}{P^{h,V}_tf(x_0)} \right)\right)^+\right]$. See Corollary \ref{cor:bnds}.
From this one obtains the following type of upper bound for  $| \nabla \log p^{h,V}_t |_{x_0} $:
$$
\frac{C_1(t)}{\sqrt{t}}\left(\sup_{y\in M} \left(\log  \frac{p^h_t(y,y_0)}{p^h_{2t}(x_0,y_0)}+ 2t(\sup V-\inf V)    \right)^+\right)^{\frac{1}{2}}+tC_2(t)|\nabla V|_{\infty}.$$
Further estimates can then be obtained from bounds on the kernels $p_t^h$. For manifolds with a pole $y_0$ we would not need bounds on $p_t^h$ and furthermore we have a formula for $\nabla \log p^{h,V}_t (\cdot, y_0)$.

\subsection{Main results  for manifolds with a pole}
Suppose that $M$ is a manifold with a pole $\pole$,  by which we mean that the exponential map $\exp_\pole$ is a diffeomorphism,  it makes sense to compare the Feynman-Kac kernel and its derivatives with the `Gaussian kernel'. Indeed, we will also present explicit formulas for  $p_t^{h,V}(\cdot, \pole)$, for
$\nabla p_t^{h,V}(\cdot, \pole)$ and for $\nabla \log p_t^{h,V}(\cdot, \pole)$ in terms of the semi-classical bridges, with terminal value  $y_0$ at a terminal time $T>0$. The semi-classical bridges are constructed from the `Gaussian kernel' and so the formulas are explicit. 
 
Let $J_{y_0}$ denote the Jacobian determinant of the exponential map at $y_0$
and let 
\begin{equation}\label{phi}\Phi(y)=\frac{1}{2}J_{y_0}^{\frac{1}{2}}(y)\Delta J_{y_0}^{-\frac{1}{2}}(y).
\end{equation} The subscript $y_0$ will be omitted from time to time.
A {\it semi-classical bridge} $\tilde{x}_s$ with terminal value $\pole$ and terminal time $T$
is a time dependent diffusion with generator $\frac{1}{2}\triangle + \nabla \log k_{T-s}(\cdot, y_0)$  where, for $d$  the Riemannian distance function,
\begin{equation}\label{kt}k_t(x_0, y_0):=(2\pi t)^{-\f n 2}e^{-\f {d^2(x_0, y_0)}{2t}}J^{-\f 12}(x_0).
\end{equation}
We refer  $(2\pi t)^{-\f n 2}e^{-\f {d^2(x_0, y_0)}{2t}}$ as the `Gaussian kernel'.
A useful feature of the semi-classical bridge is that its radial process $r_t=d(\tilde x_t, y_0)$ agrees with that of an $n$-dimensional Euclidean Brownian bridge, so $r_t$ is in fact an $n$-dimensional Bessel bridge. On $\R^n$, the semi-classical  bridge and the classical Brownian bridge (the conditioned Brownian motion) agree.  Elworthy and  Truman, whose considerations come from classical mechanics and semi-classical limits, proved the following formula (see  \cite{Elworthy-Truman-81}):
\begin{equation}\label{elaform}
p_T^V(x_0,y_0) = k_T(x_0,y_0) \E \left[ e^{\int_0^T \left(\Phi-V\right)(\tilde{x}_s)\,ds} \right].
\end{equation}
Since $\Phi$  is bounded on manifolds of constant negative curvature, (\ref{elaform}) leads to a Gaussian upper bound for the integral kernel $p_t^h$.
Further Gaussian estimates on the heat kernels are obtained by Ndumu \cite{Ndumu91} and Aida \cite{Aida-Gaussian-estimates}. An extension of the formula was given by Watling \cite{Watling1} where the method is similar to the original approach 
\cite{Elworthy-Truman-81, Elworthy-book, Elworthy-Truman-Watling}.  

Or main contribution is to use semi-classical bridges  for obtaining first order estimates. We first prove an extension of (\ref{elaform}), using a method different from that  in \cite{Elworthy-Truman-81},  and
then  deduce a first order formula.  
We also allow an additional non-gradient drift $Z$,  for which we follow a beautiful idea in \cite{Watling1} and replace the semi-classical bridge  by the `semi-classical Riemannian bridge', i.e. a Markov process with the Markov generator: 
$$\frac{1}{2}\triangle + \nabla \log k_{T-s}(\cdot, y_0)+\nabla S, \quad \hbox{where
$S(x)=\int_0^1 \<\dot \gamma(r), Z(\gamma(r))\>dr$},$$ 
and  $\gamma: [0,1]\to M$ is the unique geodesic from $x$ to $y_0$. The value $S(x)$ represents the path average of the radial part of $Z$ along $\gamma$. 
Let us define $$\Phi^h= \Phi -\f 1 2\Delta h - \f 1 2 |\nabla h|^2, \quad \Psi=
 \f 12 |\nabla S|^2+\f 12 \Delta S+\<Z, \nabla S-\nabla h\>.$$
If the Brownian motion $x_t$ with drift  $Z+\nabla h$ is complete then  the probability distribution of  $x_t$ and that of  the semi-classical Riemannian bridge are equivalent on $\F_t$ where $t<T$. Denote by $M_t$ the Radon Nikodym derivative on $\F_t$. Then 
$$M_t:=\f {k_\term (x_0)e^{(S-h)(x_0)} } {k_{\term-\tm} (\tilde{x}_\tm) e^{(S-h)(\tilde x_t)}}\exp\left[  \int_0^\tm (\Phi^{h}+\Psi)(\tilde{x}_\stm) \,d\stm\right],$$
and $(M_t, t<T)$ is a martingale. See  Lemma \ref{lemma-Girsanov}.

Let $\tilde u_t$ denote the solution to the stochastic differential equation (\ref{horizontal-2}) on the orthonormal frame bundle with the initial value $\tilde u_0$ in 
 the fibre $ \pi^{-1}(x_0)$. Set $\tilde x_t = \pi(\tilde u_t)$, then $(\tilde x_t)$  is a semi-classical bridge. 
 Set $$d\tilde{B}_r = dB_r + \tilde{u}_r^{-1}\nabla \log(e^{-h} k_{T-r})(\tilde{x}_r)\,dr.$$
 For the probabilistic representations of the derivative of the semi-group $P_t^{h, Z, V}$ we would have  expected to use the damped parallel translation equation involving the vector field $\nabla_\cdot d \log k_{T-t}$, but  we only need to involve  the following  equation,
 \begin{equation*}
{D \over dt }\tilde{W}_t=-{1\over 2}  {\Ric}_{\tilde x_t}^{\#}(\tilde{W}_t)+\Hess(h)({\tilde W_t}), \quad \tilde W_0 = {\id}_{T_{x_0}M},
\end{equation*}
whose solution is bounded under the condition that $\rho^h$ is bounded from below.
Let $\beta_t^h$ denote the exponential of the path integral 
$$\beta^h_t=\exp\left(  \int_0^t (\Phi^h+\Psi-V)(\tilde{x}_s)ds\right).$$ 

\begin{theorem*}
[ \ref{eltruform}]  Assume that $y_0$ is a pole for $M$. Suppose that the BM with drift $\nabla h+Z$ does not explode and that  $V$ belongs to $ C^{0,\alpha}\cap L_\infty$. Suppose that $\beta_t^h$ converges to $\beta_T^h$ in $L^1$, which holds if $\Phi^h+\Psi-V$ is bounded above.  Then
\begin{equation}\label{elaform-1-1} 
\f{p^{h,Z,V}_T(x_0,y_0) }{k_T(x_0, y_0)}=\f{e^{(h-S)(y_0)}}{e^{(h-S)(x_0)}}\E \left[ \exp\left(\int_0^T (\Phi^h+\Psi-V)(\tilde{x}_s)ds\right) \right].
\end{equation}
\end{theorem*}

We take $Z=0$ in the statement below, in this case $\Psi\equiv 0$. For the proof we relied on results in section 2, where the commutation relations concerning $d$ and the semi-groups generated by $\Delta^h$, on both functions and on differential forms, are used.  A non-symmetric first order term $Z$ can be added in,  once these commutative relations are proved rigorously.
\begin{theorem*}
[\ref{thm:fof}] Assume that $y_0$ is a pole for $M$. Suppose  that $\rho^h$ is bounded from below and that $\Phi^h-V$ is bounded above. Let $Z^h_t$  denote the normalised random variable $\beta_t^h/\E[\beta_t^h]$.
If $V$ is bounded and H\"{o}lder continuous then
\begin{equs}
d \left( {\log p^{h,V}_T}(\cdot,y_0) \right)_{x_0}
=&\frac{1}{T} \E \left[ Z_T^h\int_0^T \< \tilde{W}_s(\cdot),\tilde{u}_s d\tilde B_s\rangle\right]
 \\
&+ { \E \left[ Z_T^h \int_0^T V (\tilde x_{T-s})  \;e^{-\int_{T-s}^T V(\tilde x_u)du }
\f  {1}{T-s}\int_0^{T-s} \< \tilde W_r(\cdot),\tilde u_r d\tilde B_r\> \right]}.
\end{equs}
If $V\in BC^1$, then
$$\begin{aligned}
\f{d(p^{h,V}_T(\cdot,y_0))(x_0)}{k_T(x_0, y_0)} &=e^{h(y_0)-h(x_0)} \E \left[ e^{\int_0^T (
\Phi^h-V)(\tilde x_s) ds}  \f 1 T \int_0^T \< \tilde{W}_r(\cdot),\tilde{u}_r d\tilde{B}_r-(T-r)\nabla Vdr\>\right].\end{aligned}$$
\end{theorem*}
Since $|\tilde W_r|$ is bounded, explicit estimates for $ {|\nabla \log p^{h,V}_T(\cdot,y_0)|_{x_0}} $ follow from these formulas. For example the following type of estimate holds :
\begin{corollary*} [\ref{pole-estimate}] Assume that the relevant data are bounded,
$$\begin{aligned}|\nabla p^{h,V}_T(\cdot,y_0)|_{x_0}\le &Ce^{h(y_0)-h(x_0)} |\beta^h_T|_\infty (2\pi t)^{-\f n 2}e^{-\f {d^2(x_0, y_0)}{2t}}J^{-\f 12}(x_0)\\
&\cdot\left(\f  {d(x_0, y_0)}T+ |\nabla h|_\infty+1+ \f 1 {\sqrt T}+T|dV|_\infty\right).\end{aligned}$$
\end{corollary*}
Precise Gaussian estimates and formulas  for heat kernels and their derivatives using semi-classical bridge were obtained by  Aida \cite{Aida01} for asymptotically flat Riemannian manifolds with a pole where  the derivatives of $\log J$ up to order 4, the Riemannian curvature and the derivative of the Ricci curvature are assumed to be bounded. Our conditions, see Corollary \ref{pole-estimate} in section 3, are  weaker.

\section{First order Feynman-Kac formula}\label{section2}

Let $\wedge$ denote the space of differential forms and we denote by $d^*$  the $L_2$ adjoint of the exterior differential  operator $d$ with respect to the measure $e^{2h}dx$ where $dx$ is the volume measure.
Then $d+d^*$ and all its powers are essentially self-adjoint on $L^2(\wedge, e^{2h}dx)$ and we denote by the same notation their closures, see \cite{Li-thesis,Li-Myers}. Let $\Delta^h$ denote the weighted Laplace-Beltrami operator $\Delta^h=(d+d^*)^2$ with its initial domain  smooth and compactly supported differential forms. 

For $k \in 0,1,\cdots$ denote by $H^k$ the completion of $H^k_0$, where 
\begin{equation*}
H^k_0 = \bigg\lbrace f \in C^\infty: |f|^2_{H^k} =  \sum_{j=0}^k |\nabla^{(j)}f|^2_{L_2}  < \infty\bigg\rbrace,
\end{equation*}
under the norm $|\cdot |_{H^k}$. Denote by $\overline{C_K^\infty}^{H_k}$ the closure of $C_K^\infty$ under $|\cdot |_{H^k}$. Denote also by $d^*$ the dual of $d$ in $L_2(M; dx)$, taking $h=0$,  the latter having initial domain $C^\infty_K$. Then the Laplace-Beltrami operator on functions is the closure of $-d^* d$ and for any complete Riemannian manifold $\Dom(d) = \overline{C^\infty_K}^{H^1} = H^1$. For higher order derivatives the corresponding statements  hold for manifolds with bounded geometry, see \cite{Aubin}. For $k=2$, $\overline {C_K^\infty}^{H^2}=H^2$ if the injectivity radius of $M$ is positive and if the Ricci curvature is bounded below, see \cite{Hebey}.  We avoid these assumptions.

\subsection{Basic formulas}
In this section we obtain the basic formulas for $dP_t^{h,V}$ which we will use to obtain estimates and to obtain a representation for $dp_t^{h,V}$. The usual differentiation formulas concerning 
$dP_t^{h}f$ relies on the intertwining property that $d$ commutes with $P_t^h$, which follows from the essential self-adjointness of $d+d^*$ and whose proof can be found in~\cite{Li-thesis}. 

Suppose that $V$ is bounded, then the operator $f\to Vf$ is $\Delta^h$ bounded and  by the Kato-Rellich theorem, $\f 12\Delta^h-V$ 
is self-adjoint on the domain of $\Delta^h$ and essentially self-adjoint on $C_K^\infty$, for more general potentials see \cite{Oleinik, Shubin-01}. 
By functional calculus $e^{-t(\f 12 \Delta^h-V)} $ is a strongly continuous contraction semi-group on $L_2(M; e^{2h}dx)\cap L_b$, where the $L_2$ space is defined by the weighted measure $e^{2h}dx$. Furthermore if $f \in L_2(M; e^{2h}dx)$ then  $e^{-t(\f 12 \Delta^h-V)} f$ belongs to the domain of $ \Delta^h$.
By direct computation $\E \left(f(x_t) \V_t\right)$, where $x_t$ is a Brownian motion with drift $\nabla h$ and $\V_t = e^{-\int_0^t V(x_r) dr}$, is also a strongly continuous contraction semi-group on $L_2\cap L_\infty$ with generator $\f 12 \Delta^h-V$ and core $C_K^\infty$. Consequently
$$e^{-t(\f 12 \Delta^h-V)} f=\E \left(f(x_t) \V_t\right), \qquad f\in L_2\cap L_\infty.$$
Furthermore $g_t:=e^{-t(\f 12 \Delta^h-V)} f$ satisfies the variation of constant formula
$$g_t=P_t^hf+\int_0^t P^h_{t-s}(Vg_s)\, ds,$$
and consequently they are $C^{1,2}$ functions and solve the parabolic equation. The solution measure has a density $p^{h,V}_t$ with respect to the volume measure.

 The h-Brownian motion we use will be given by the canonical construction  below.
Let $\lbrace B_t^i \rbrace_{i=1}^n$ be a  family of independent one-dimensional Brownian motions and set $B_t=(B_t^1, \dots, B_t^n)$.
Let $\lbrace H_i \rbrace$ be canonical horizontal vector fields on the orthonormal frame bundle $OM$ of $M$, associated to an orthonormal basis of $\R^n$.  Let $\pi$ be the projection that takes a frame at $x\in M$ to $x$ and for  $u\in \pi^{-1}(x)$ let $\h_u$ denote the horizontal lift from the tangent space $T_xM$ at $x$ to the tangent space $T_uOM$ at $u$.
If the $h$-Brownian motion is complete (non-explode),  then the following stochastic differential equation (SDE) is complete, 
\begin{equation}\label{horizontal}
du_t = \sum_{i=1}^n H_i(u_t) \circ dB^i_t+\h_{u_t} \left({\nabla h}\right)\; dt.
\end{equation}
Furthermore  $x_t := \pi (u_t)$ is a h-Brownian motion on $M$ starting at $x_0:=\pi(u_0)$.  If $\Ric-2\Hess (h)$ is bounded from below, then the $h$-Brownian motion is complete.  If $h$ vanishes it is sufficient to assume that $\inf_{v\in T_xM: |v|=1}\{\Ric_x(v,v)\}\ge - \alpha(r(x))$ where $\alpha$ grows at most quadratically. 
 
Let us consider the damped parallel transport equation along $x_t$
\begin{equation}\label{damped}
{D \over dt }W_t=-{1\over 2}  {\Ric}_{x_t}^{\#}(W_t)+\Hess(h)({W_t}), \quad W_0 = {\id}_{T_{x_0}M}
\end{equation}
and denote by $W_t$ its solution.  The notation ${D\over dt}W_t$ denotes the covariant derivative  ${D\over dt}W_t = u_t \frac{d}{dt} u_t^{-1} W_t$ and so equation (\ref{damped}) is interpreted as follows: 
${u}_t^{-1}{W}_t$ solves the equation
$
\f d {dt} w_t = -\frac{1}{2}({u}_t^{-1}{\Ric}^\sharp_{{x}_t}{u}_t)  w_t+ {u_t}^{-1} \Hess(h)( u_t  w_t)$,
with initial value $ w_0 ={ \id}_{\R^n}$.

We need the following formulation for the Feynman-Kac formula and the differentiation formula \cite{Li-thesis}, whose elementary proofs are included for completeness. We recall the notation $\rho^h(x)=\inf_{v\in T_xM, |v|=1}\{\Ric(v,v)-2\Hess(h)(v,v)\}$. If $\rho^h\ge K$ then $|W_t|\le e^{-Kt}$.

\begin{lemma}\label{lem-one} \begin{enumerate}
\item[(a)] Suppose that the $h$-Brownian motion $(x_t)$ does not explode and $V$ is bounded from below.
 If $P_t^{h,V}f$ is a $C^{1,2}$ solution to the parabolic equation, then for any $T>0$,
\begin{equation}\label{eqn-lemma1}
\V_s P^{h,V}_{T-s}f(x_s) = P^{h,V}_Tf(x_0) + \int_0^s \V_r dP^{h,V}_{T-r}f(u_r \,dB_r), \quad 0\le s \le T.
\end{equation}
\item[(b)] Suppose that $\rho^h$ is bounded from below (in particular  the $h$-Brownian motion does not explode ) and $f\in L_\infty$, then for  $0<t<T$,
\begin{equation}\label{differentation}
\E \left(P^{h}_{T-t}f(x_t) \int_0^t\left \<u_r\, dB_r, W_r(v)\right\> \right)  =t \,dP_T^hf(v).\
\end{equation}

\end{enumerate}\end{lemma}

\begin{proof}
(a) By the assumption on the $h$-Brownian motion, $u_s$ exists for all time. Since $V$ is bounded from below,  the solution to the equation $\f d{ds}g_s = -V(x_s) g_s$ does not explode and is satisfied by $\V_s$.
We may therefore apply It\^o's formula to the pair of stochastic processes $x_s, \V_s$ using the identity $dx_s = u_s \circ dB_s+\nabla h(x_s)\,ds$ to obtain (\ref{eqn-lemma1}).

(b) Suppose that $V=0$ and  $\rho^h\ge K$ where $K$ is a constant, it is sufficient to prove the conclusion for $f\in BC^1$. In fact  using the  unit speed geodesic from $x$ to $y$ and formula (\ref{differentation}) we conclude the following :
$$|P_t^hf(x)-P_t^hf(y)|\le C(t)|f|_\infty d(x,y), \qquad f\in BC^1$$
where $d(x,y)$ denotes the distance between $x$ and $y$. Let us denote
by $P^h(t,x,\cdot)$ (and $P^h(t, y,\cdot)$ respectively)  the  transition probability from $x$ (from  $y$). Then the total  variation norm,  $|P^h(t,x,\cdot)-P^h(t, y,\cdot)|_{TV}$ is controlled by $C(t) d(x,y)$. Consequently, $P_t^hf$ belongs to $BC^1$ whenever $f$ is bounded. For any $t<T$, setting $\tilde T=(T+t)/2$ and $\tilde f=P^h_{(T-t)/2} f$,   we then apply the formula  to $\tilde f$ to obtain a formula for  $t \, d P_{\tilde T} \tilde f$, (\ref{differentation}) follows from the semigroup property.  

We return to the proof of  (\ref{differentation}) where $f$ is a $BC^1$ function. Since $P_t^hf$ is $C^{1,2}$ and
$\int_0^t\<W_s(v), u_s dB_s\>$ is an $L^2$ martingale, with which we
multiply both sides of (\ref{eqn-lemma1}).   Use the It\^o isometry we obtain
$$\begin{aligned}
&\E \left( \int_0^t \<u_rdB_r, W_r(v)\> P^{h}_{T-t}f(x_t)\right)  \\
&= \E \left(\int_0^t dP^{h}_{T-r}f(u_r dB_r)\int_0^t \<u_rdB_r, W_r(v)\>\right)=t \,dP_T^hf(v).
\end{aligned}$$
We have used the Markovian property and that $\E[  df(W_r(v))]=dP_T^hf(v)$. The proof is complete.
 \end{proof}

\begin{theorem}
\label{Feynman-Kac-kernel-1}
Assume that $\rho^h$ is bounded from below and $V$ is a bounded H\"{o}lder continuous function.
Then for all $f\in L_\infty$ and $v\in T_{x_0}M$,
$$\begin{aligned}
(dP^{h,V}_tf)(v)= &\frac{1}{t} \E \left[ f(x_t) \int_0^t \< W_s(v),u_s dB_s\>\right]\\
&+ \E \left[   f( x_t) \left(\int_0^t e^{-\int_{t-s}^t V(x_u)du}
\f {V (x_{t-s})} {t-s}\int_0^{t-s} \< W_r(v),u_r dB_r\>\right)ds\right].\end{aligned}$$
It follows that $P_t^{h,V}f\in BC^1$  for any $0< s \le t$ and
 the following Feynman-Kac formula holds: 
 $P^{h,V}_tf(x_0)=\E\left [\V_s P^{h,V}_{t-s}f(x_s) \right]$.
\end{theorem}

\begin{proof}
If $f\in  L_\infty\cap L_2$ so does $VP^{h,V}_sf $, then $P^h_{t-s}(VP^{h,V}_sf)$ is smooth.
 Since  $\rho^h\ge K$ for a constant $K$,   apply  part (b) of Lemma  \ref{lem-one} to obtain
$$d(P_{t-s}^h (VP^{h,V}_sf) )(v)= \f 1 {t-s} \E \left[  (V P^{h,V}_sf)(x_{t-s}) \int_0^{t-s} \< W_r(v),u_r dB_r\>\right].$$
By the Feynman-Kac formula, $P_t^{h,V}f=\E \left(f(x_t)
\V_t\right)$ is bounded, so $$\left|d(P_{t-s}^h (VP^{h,V}_sf) )(v)\right|\le \f 1 {t-s} |V|_\infty \left| P^{h,V}_sf\right|_\infty \sqrt{ \int_0^{t-s}\E |W_r(v)|^2 dr}.$$
From the equation for $W^h_t$ we also have the inequality:  $|W_t^h(v)|\le |v| e^{-\f 12 \int_0^t \rho^h(x_s) ds} $. Since  $\rho^h$  is bounded from below, there exists a constant $C$ depending on $|V|_\infty$ and $|f|_\infty$ such that 
 $\left|d(P_{t-s}^h (VP^{h,V}_sf) )(v)\right|\le 
C\f 1 {\sqrt{t-s}}$.
 We may differentiate both sides of the variation of constants formula $P^{h,V}_tf=P_t^hf+\int_0^t P^h_{t-s}(VP^{h,V}_sf)\, ds$ whose last two terms  is differentiable and so is $P_t^{h,V}f$.
For any $v\in T_{x_0}M$, we have\begin{equs}
dP^{h,V}_tf(v)
=&dP^h_tf(v)+ \int_0^t \f 1 {t-s} \E \left[  (V P^{h,V}_sf)(x_{t-s}) \int_0^{t-s} \< W_r(v),u_r dB_r\>\right] ds.\end{equs}
By  the Markov property  we obtain the identity\begin{equs}
&\E \left[  (V P_s^{h,V}f)(x_{t-s}) \int_0^{t-s} \< W_r(\cdot),u_r dB_r\>\right]\\
&=\E \left[  V (x_{t-s}) f( x_t) e^{-\int_{t-s}^t V(x_u)du}
\int_0^{t-s} \< W_r(\cdot),u_r dB_r\>\right].
\end{equs}
If $f\in L_\infty$, we take an approximating sequence to conclude the required formula, from which and that $|W_t|\le e^{-Kt}$ for a constant $K$ we see that $P_t^{h,V}f$ is $C^1$ with bounded derivative. By the method in \cite{Li-Hessian} we can prove directly that $P_t^{h,V}$ is $C^2$, but here we borrow from  elliptic regularity.
We take expectations of both sides of (\ref{eqn-lemma1}), whose stochastic integral term $\int_0^s \V_r dP^{h,V}_{t-r}f(u_r \,dB_r)$ is an $L^2$ bounded martingale,  to  conclude 
that $\E [\V_s P^{h,V}_{t-s}f(x_s) ]= P^{h,V}_tf(x_0)$ and by taking $s\to t$ we see that 
$$ P^{h,V}_tf(x_0)=\E[\V_tf(x_t)],$$
and that $P_t^{h,V}f$ is bounded, concluding the proof.
\end{proof}

\subsection{Estimates for general manifolds}
In this section we establish another differentiation formula for a more regular potential and use it for obtaining gradient estimates.
\begin{lemma} \label{lemma3-V}
Suppose that the $h$-Brownian motion $(x_t)$ does not explode and $T$ is a positive number. Assume that $V$ is bounded and in $C^{1,\alpha}$. If $f \in L_\infty\cap C^1$ then for all $0\leq t<T$ and $v\in T_{x_0}M$ we have
\begin{equation}\label{eqn-lemma3}
\begin{aligned}
\V_t \cdot \left( dP^{h,V}_{T-t}f\right) (W_t(v)) = \,&dP^{h,V}_T f(v)+\int_0^t \V_s\cdot  \left(\nabla dP^{h,V}_{T-s}f\right) (u_s dB_s,W_s(v))\\
&+ \int_0^t \V_s \cdot dV(W_s(v))\cdot  P^{h,V}_{T-s} f (x_s)\,ds.
\end{aligned}
\end{equation}
\end{lemma}

\begin{proof}
Since $f, V \in C^{1, \alpha}$, the solution $P^{h,V}_t f$ is three times differentiable in space and we may differentiate both sides of the parabolic equation. Since $P^{h,V}_t f \in \Dom(d)$, it follows that $d\Delta (P^{h,V}_t f) = \Delta^{1,h}d(P^{h,V}_t f)$, 
see \cite [Chapter 2]{Li-thesis} or \cite{Gaffney} the case for $h=0$.
Consider the function $(t,\alpha,(x,v)) \in [0,T]\times\R_+ \times TM \mapsto (\alpha,dP^{h,V}_{T-t}f(v))$ and apply  It\^{o}'s formula to it and  to the process $(t, \V_t, W_t(v))$ to obtain
\begin{equation}\label{eqn-prooflem3}
\begin{aligned}
&\V_tdP^{h,V}_{T-t}f (W_t(v))-dP^{h,V}_Tf(v)\\
=\text{ }& \int_0^t \V_s \nabla (d P^{h,V}_{T-s}f)(u_sdB_s, W_s(v))+\int_0^t \nabla  \V_s(d P^{h,V}_{T-s}f)(\nabla h(x_s), W_s(v))ds\\
&+\int_0^t  \V_s  \left( {\frac{\partial}{\partial s}} dP^{h,V}_{T-s}f \right)(W_s(v))\,ds+\frac{1}{2}\int_0^t \V_s \tr \nabla^2 (dP^{h,V}_{T-s}f)(W_s(v))\,ds\\
&+\int_0^t \V_s dP^{h,V}_{T-s}f \left({\frac{D}{ds}} W_s(v)\right)ds
-\int_0^t V (x_s)\V_s dP^{h,V}_{T-s}f (W_s(v))\,ds.
\end{aligned}
\end{equation}
Using Bochner's formula, $\Delta^{1,h} = \trace \nabla^2 +2L_{\nabla h}-{\Ric}^\sharp$ for the Laplace-Beltrami operator $\Delta^1$ on differential $1$-forms,  the definition of the Lie derivative and  equation (\ref{damped}), we thus have
\begin{equation*}
\begin{aligned}
&\V_tdP^{h,V}_{T-t}f (W_t(v))- dP^{h,V}_Tf(v)\\
=&\int_0^t \V_s \nabla (d P^{h,V}_{T-s}f)(u_sdB_s, W_s(v))-\int_0^t V (x_s)\V_s dP^{h,V}_{T-s}f (W_s(v))\,ds\\
&+\int_0^t  \V_s  \left( {\frac{\partial}{\partial s}} dP^{h,V}_{T-s}f \right)(W_s(v))\,ds+\frac{1}{2}\int_0^t \V_s \Delta^{1,h}(dP^{h,V}_{T-s} f)(W_s(v))\,ds.
\end{aligned}
\end{equation*}
 We can commute the time and space derivatives and also commute $\Delta^h$ with $d$  to obtain
\begin{equation*}
\frac{\partial}{\partial s}d(P^{h,V}_{T-s}f) + \frac{1}{2}\Delta^{1,h} (dP^{h,V}_{T-s}f) = Vd(P^{h,V}_{T-s}f) + (dV) P^{h,V}_s f.
\end{equation*}
By substituting the into the earlier equation we then obtain, after some cancellation, the required formula.
\end{proof}

\begin{proposition}\label{lem-four}
Suppose that $\Ric-2\Hess (h)\ge K$  and that $V \in  BC^1\cap C^{1, \alpha}$. Then for all $0\le t < T$, $v\in T_{x_0}M$ and $f \in L_\infty\cap C^1$, we have the formula
\small
\begin{equation}\label{eqn-five}
\E\left[ \V_t(dP^{h,V}_{T-t} f)(W_t(v))\right] = d(P^{h,V}_T f)(v) + \E \left[ \int_0^t \V_s dV(W_s(v)) P^{h,V}_{T-s}f(x_s)\,ds\right].
\end{equation}
\end{proposition}

\begin{proof}
We prove that the stochastic integral in the equation (\ref{eqn-lemma3}) is a martingale with vanishing expectation. By Theorem \ref{Feynman-Kac-kernel-1}, $dP_tf$ is bounded. By the boundedness of $|W_t(v)|$, we see that $ \V_t(dP^{h,V}_{T-t} f)(W_t(v))$ is bounded and so is
$ \int_0^t \V_s dV(W_s(v)) P^{h,V}_{T-s}f(x_s)\,ds$. Hence the stochastic integrale is a bounded martingale.
\end{proof}

\begin{theorem}\label{thm:maintheorem} 
Assume that $V \in BC^1$ and $f\in L_\infty$.
 Suppose that $\Ric-2\Hess(h)$ is bounded from below. Then for all $0< t \leq T$ and $v\in T_{x_0}M$ we have
\begin{equation}
\begin{aligned}
dP^{h,V}_Tf(v) =\text{ }& \frac{1}{t}\E \left[ \V_T f(x_T) \int_0^t \< W_s(v),u_s dB_s\>\right]\\
&-\frac{1}{t} \E \left[\V_Tf(x_T)\int_0^t \int_0^r  dV(W_s(v)) \,ds\,dr\right].
\end{aligned}\nonumber
\end{equation}

\end{theorem}

\begin{proof}
 By Lemma \ref{lem-one} we have
\begin{equation}\label{eqn-seven}
\V_t P^{h,V}_{T-t}f(x_t) = P^{h,V}_T f(x_0) + \int_0^t \V_r d(P^{h,V}_{T-r}f)(u_r\,dB_r).
\end{equation}
Next we multiply the above equation by the $L_2$ martingale $\int_0^t \<u_s dB_s,W_s(v)\>$ and use It\^{o}'s isometry to obtain
\begin{equation*}
\E\left[ \V_t P^{h,V}_{T-t}f(x_t) \int_0^t \<u_s dB_s,W_s(v)\> \right]= \E\left[ \int_0^t \V_r d(P^{h,V}_{T-r} f)(W_r(v))\,dr\right]
\end{equation*}
using the fact that the last term in equation (\ref{eqn-seven}) is an $L_2$ martingale. Suppose that $V$ is $C^2$ and $f\in C^1$, we apply equation (\ref{eqn-five}) to deduce that for any $0< t\leq T$,
\begin{equation}\label{eqn-proponeproof}
\begin{aligned}
dP^{h,V}_Tf(v) =\text{ }& \frac{1}{t}\E \left[ \V_t P^{h,V}_{T-t}f(x_t)\int_0^t \< u_r dB_r,W_r(v)\>\right]\\
&-\frac{1}{t}\E\left[\int_0^t \int_0^r \V_s dV(W_s(v))P^{h,V}_{T-s}f(x_s)\,ds\,dr\right]
\end{aligned}
\end{equation}
and the required identity follows from the Markov property. Next we approximate $V\in BC^1$ by smooth functions and $f$ by smooth functions to conclude.
\end{proof}

 For manifolds whose second fundamental form of an isometric embedding satisfies suitable conditions, the formula can be deduced 
 from that in \cite{Li-thesis, Elworthy-Li} using the techniques of filtering our redundant noise. Here we do not pose conditions on the second fundamental form.

For estimations we need the following lemma, which is \cite[Lemma 6.45]{Stroock2000}.

\begin{lemma}\label{Stroock's-Lemma}
Suppose $(\Omega,\mathcal{F},\mathbb{P})$ is a probability space and $\phi\geq 0$ is a measurable function on $\Omega$ with $\E[\phi]=1$. If $\Psi$ is a measurable function on $\Omega$ such that $\phi \Psi$ is integrable then
\begin{equation}
\E\left[\phi \Psi\right] \leq \E\left[\phi \log \phi\right] + \log \E \left[ \exp (\Psi)\right].\nonumber
\end{equation}
\end{lemma}

\begin{proposition}\label{prop:derivativeestiamte}
Suppose that $\Ric - 2\Hess (h) \geq K$ and $V \in  BC^1$. Then for a non-negative bounded measurable function $f$ with $P_t^{h,V}f(x_0)=1$ we have
\begin{equation}\label{eq:gradbnd}
| \nabla P^{h,V}_t f|_{x_0} \leq \frac{1}{\sqrt{t}}\Big(2C_1(t,K) \E\left[(f(x_t)\V_t\log (f(x_t)\V_t))^+\right] \Big)^{\frac{1}{2}}+tC_2(t,K) |\nabla V|_{\infty} \nonumber
\end{equation}
for all $t>0$ where
\begin{equation*}
C_1(t,K) := \frac{1-e^{-Kt}}{Kt}, \quad C_2(t,K) := \frac{2}{Kt}\left(1+\left(\frac{e^{-Kt/2}-1}{Kt/2}\right)\right).
\end{equation*}
We observe that
$\lim_{t\to 0} C_1(t,K) =1$ and $\lim_{t\to 0} tC_2(t,K) = \frac{4}{K}$.
\end{proposition}

\begin{proof}
For $\gamma \in \mathbb{R}$ set
\begin{equation}
\phi := f(x_t)\V_t,\quad \psi_t := \gamma \int_0^t \langle W_s(v_0),u_s dB_s-(t-s)\nabla V(x_s)\,ds\rangle\nonumber
\end{equation}
to see, by Theorem \ref{thm:maintheorem}, Fubini's theorem and Lemma \ref{Stroock's-Lemma}, that for $v_0\in T_{x_0}M$,
\begin{equation}
\begin{aligned}
\gamma t d(P^{h,V}_t f)(v_0) & \leq \E\left[\phi\log \phi\right]+\log \E\left[\exp\left[\gamma \int_0^t \langle W_s(v_0),u_s dB_s-(t-s)\nabla V(x_s)\,ds\rangle\right]\right].
\end{aligned}\nonumber
\end{equation}
Since $\frac{D}{ds}W_s= -\frac{1}{2}{\Ric}^\sharp (W_s)+\nabla^2 h(W_s, \cdot)^\sharp$ we have $|W_s| \leq e^{-\frac{K s}{2}}$ so, 
\begin{equation*}
\begin{aligned}
& \text{ }\log \mathbb{E}\left[\exp\left[\gamma \int_0^t \langle W_s(v_0),u_s dB_s-(t-s)\nabla V(x_s)ds\rangle\right]\right]\\
\leq & \text{ }\log \mathbb{E}\left[\exp\left(\gamma\int_0^t \langle W_s(v_0),u_s dB_s\rangle +|\;\gamma|  \left|\int_0^t (t-s)\langle -\nabla V(\xi_s),W_s(v_0)\rangle ds \right|\right)\right]\\
\leq & \text{ } {\gamma^2 } \int_0^t e^{-Ks} |v_0|ds +  |\gamma |  |\nabla V|_{\infty}\int_0^t (t-s) e^{-\frac{Ks}{2}}  |v_0| ds \\
\leq & \text{ }\gamma^2 tC_1(t, K)|v_0| + |\gamma| |\nabla V|_{\infty} {t^2}C_2(t, K)|v_0|.
\end{aligned}
\end{equation*}
Thus we obtain
\begin{equation*}
\gamma t d(P^{h,V}_t f)(v_0) \leq  \E\left[(\phi\log \phi)^+\right]
 + \gamma^2 tC_1(t,K) |v_0|+ |\gamma| |\nabla V|_{\infty} {t^2}C_2(t, K)
\end{equation*}
which after minimising over $\gamma$ yields
\begin{equation}
t| \nabla P^{h,V}_t f|_{x_0}\leq \left(2tC_1(t,K) \E\left[(\phi\log \phi)^+\right] \right)^{\frac{1}{2}}+t^2|\nabla V|_{\infty} C_2(t,K),\nonumber
\end{equation}
as claimed.
\end{proof}

If $f\ge 0$ and $P^{h,V}_tf(x_0)>0$ we define 
\begin{equation*}
{\mathcal H}_t(f,x_0):= \mathbb{E}\left[
\frac{f(x_t)\mathbb{V}_t}{P^{h,V}_tf(x_0)} \left( \log\left( \frac{f(x_t)\mathbb{V}_t}{P^{h,V}_tf(x_0)} \right)\right)^+\right],
\end{equation*}
and replacing the non-negative function $f$ in Proposition \ref{prop:derivativeestiamte} by $\f f {P^{h,V}_tf(x_0)}$,
 we deduce the following corollary.

\begin{corollary}\label{cor:bnds}
Suppose that $\Ric - 2\Hess (h) \geq K$ and $V \in  BC^1$.  Then for any non-identically zero bounded measurable non-negative function $f$ we have
\begin{equation}
| \nabla \log P^{h,V}_t f |_{x_0} \leq \frac{1}{\sqrt{t}}\Big(2C_1(t,K) {\mathcal H}_t(f,x_0) \Big)^{\frac{1}{2}}+tC_2(t,K) |\nabla V|_{\infty} , \quad t>0.\nonumber
\end{equation}
and 
\small
$$
| \nabla \log p^{h,V}_t |_{x_0} \leq \frac{1}{\sqrt{t}}\sqrt{2C_1(t,K) }\left(\sup_{y\in M} \left(\log  \frac{p^h_t(y,y_0)}{p^h_{2t}(x_0,y_0)}+ 2t(\sup V-\inf V)    \right)^+\right)^{\frac{1}{2}}+t\,C_2(t,K) \,|\nabla V|_{\infty}.$$
\normalsize
\end{corollary}
Indeed, choosing $f(\cdot) = p^{h,V}_t(\cdot,y_0)$, the Feynman-Kac kernel associated to  $P^{h,V}_t$, we have
\begin{equation}\label{ineq-hbndestproof}
\begin{aligned}
{\mathcal H} _t(f,x_0) &\leq \sup_{y\in M} \left(\log  \frac{f(y)
  \E \left[e^{-\int_0^t V(x_s) ds}|x_t=y \right]}   {P^{h,V}_tf(x_0)} \right)^+ \mathbb{E}\left[\frac{f(x_t)\mathbb{V}_t}{P^{h,V}_tf(x_0)} \right]\\
&\leq \sup_{y\in M} \left( \log  \frac{p^{h,V}_t(y,y_0)e^{- t\inf V}}{p^{h,V}_{2t}(x_0,y_0)}\right)^+\\
&= \sup_{y\in M}\left( \log  \frac{p^h_t(y,y_0) \mathbb{E}\left[ e^{-\int_0^t V(b^{t, y,y_0}_s) ds}\right]e^{- t\inf V}}
{p^h_{2t}(x_0,y_0) \mathbb{E}\left[ e^{-\int_0^{2t} V(b^{t, x_0, y_0}_s) ds}\right]}\right)^+
\\&\leq \sup_{y\in M}\left( \log  \frac{p^h_t(y,y_0)}{p^h_{2t}(x_0,y_0)}+ 2t(\sup V-\inf V)\right)^+\nonumber
\end{aligned}
\end{equation}
where $b_s^{t,x,y}$ is the process obtained by conditioning the $h$-Brownian motion by $x_t=y$ and $x_0=x$.

Given suitable heat kernel upper and lower bounds or more generally a Harnack inequality for $p_t^h$,
we would have an estimate on the logarithmic derivative of $p^{h,V}_t$. Heat kernel bounds and Harnack inequalities are naturally related. For the first see \cite{CLY,Li-Yau,Cheeger-Gromov-Taylor,LDUP, Davies-Gaussian-bound, Norris97-long}, \cite[Corollary 3.1]{Li-Yau}, and \cite[Theorems 7.4, 7.5, 7.9]{Grigoryan99}. If the Sobolev inequality $|f|_{2n/(n-2)}^2 \le C \int_M |\nabla f|^2\,dx$ holds for $f\in C_K^\infty$ then the heat kernel satisfies the on-diagonal estimate $p_t(x,x)\le C t^{-\f n 2}$ (leading to off-diagonal estimates).  In fact Varopoulos proved that the Sobolev inequality is also necessary for the on-diagonal upper bound, \cite[Lemma 5.1, Theorem 6.1]{Grigoryan91}. See also \cite{{Grigoryan99},Saloff-Coste} for a clear account on the relation between various functional inequalities and the on diagonal Gaussian upper bounds of the type $\vol(B(x, \sqrt t))^{-1} e^{-\f {d(x,y)^2} t}$.
For the case $h=0$ and $\Ric \geq - K$, $K\geq 0$, a global Harnack inequality exists,  \cite[Theorem 2.2]{Li-Yau}, for  positive solutions  of the heat equation. For example, \cite[Corollary~2]{Bakry-Qian}, 
\begin{equation}
\begin{aligned}
\frac{f_t(x)}{f_{t+s}(y)} \leq & \left(\frac{t+s}{t}\right)^\frac{n}{2}\exp\left[ \frac{(d(x,y)+ \sqrt{nK}s)^2}{4s} \right. \\
&\quad\quad\quad\quad\quad\quad\quad\quad\left.+ \frac{\sqrt{nK}}{2} \min\lbrace (\sqrt{2}-1)d(x,y),\frac{\sqrt{nK}}{2}s\rbrace\right].
\end{aligned}\nonumber
\end{equation}

According to \cite[Lemma 10.76]{Stroock2000}, there exists a constant $c>0$, depending only on the dimension, a lower bound on Ricci curvature and $T>0$, such that
\begin{equation}\label{eq:ratiobnd}
\frac{p_t(y,y_0)}{p_{2t}(x_0,y_0)} \leq c \exp\left[\frac{d^2(x_0,y_0)}{2t}\right]
\end{equation}
for all $t \in [0,T]$ and $y \in M$. This is obtained in the following way. First, one writes
\begin{equation}
\frac{p_t(y,y_0)}{p_{2t}(x_0,y_0)} = \frac{p_t(y,y_0)}{p_{t}(y_0,y_0)} \frac{p_t(y_0,y_0)}{p_{2t}(x_0,y_0)}.
\end{equation}
One then uses the Harnack inequality to bound the second ratio. The first can be dealt with using standard heat kernel bounds, such as those of \cite{Li-Yau}, and Bishop's volume comparison theorem. For the details, see the proof of \cite[Lemma 10.76]{Stroock2000}. Having obtained \eqref{eq:ratiobnd}, the result then follows from the second part of Corollary \ref{cor:bnds}.
From this  and the last Corollary can be deduced the following 
\begin{equation}
|\nabla \log p^{V}_{t}(\cdot,y_0)(x_0)| \leq C(T)\left(\frac{1}{t}+\frac{d^2(x_0,y_0)}{t^2}+|V|_{\infty}\right)^{\frac{1}{2}} + C(T)t|\nabla V|_\infty
\end{equation}
for all $t \in \left(0,T \right]$ and $x_0,y_0 \in M$.
When $V=0$, this recovers the estimates in  \cite{Sheu}, see also \cite{Malliavin-Stroock, HsuEstimates, Stroock-Turetsky-98,Engoulatov, XDLi-Hamilton}. 
For brevity we do not write down the estimate involving $h$ it is however worth noticing that  gradient formula for $P_t^hf$ would lead to a parabolic Harnack inequality for $p_t^h$, see \cite{Arnaudon-Thalmaier-Wang} for such an estimate.

\section{On manifolds with a pole}\label{section3}

Let $n\ge 2$. Assume that $y_0$ is a pole for $M$. That is, we assume that the exponential map $\exp_{y_0}$ is a diffeomorphism between $T_{y_0}M$ and $M$. If the Jacobian $J:M \rightarrow \mathbb{R}$ defined~by $$J(y)\equiv J_{y_0}(y)= | \det D_{\exp_\pole^{-1}(y)} \exp_\pole|$$ is non-singular, then we denote by $J^{-1}$ the reciprocal of $J$ and for $t>0$ define
\begin{equation}
\Phi(y)=\frac{1}{2}J^{\frac{1}{2}}(y)\Delta J^{-\frac{1}{2}}(y),
\qquad k_t(y) = (2\pi t)^{-\frac{n}{2}} e^{-\frac{r^2(y)}{2t}} J^{-\frac{1}{2}}(y)\nonumber
\end{equation}
for $y \in M$, where $r$ denotes the distance to the pole $y_0$. The function $J^{\f 1 2}$ is called Ruse's invariant (see for example Walker \cite{Walker}).  The objective is to obtain probabilistic representation for the kernel $p^V$ and its derivatives, involving the distance function, $J$ and $\Phi$.
On the standard $n$-dimensional hyperbolic space, the heat kernel $p_t(x,y)$ depends on $x$ and $y$ through $r =d(x,y)$ and is given by iterative formulas:
$$\begin{aligned}p_t(x,y)=&C(m)\f 1 {\sqrt t }\left(\f 1{ \sinh r}\f {\partial}{\partial r}\right) ^m e^{-m^2 t-\f {r ^2}{4t}}, \qquad\quad n=2m+1,\\
p_t(x,y)=&C(m)t^{-\f 3 2}   e^{- \f {(2m+1)^2 }{4}t}
\left(\f 1{ \sinh r}\f {\partial}{\partial r}\right) ^m\int_\rho^\infty \f {s e^{-\f {s^2} {4t} }}{\cosh s -\cosh r}\,ds, \quad n=2m+2.
\end{aligned}$$
 If its sectional curvature is $-R^{-2}$, then  \cite{Elworthy-book},
\begin{eqnarray*}
J &=& \left(\frac{R}{r}\sinh \frac{r}{R}\right)^{n-1}\\
\Phi &=& -\frac{(n-1)^2}{8R^2}+\frac{(n-1)(n-3)}{8}\left(r^{-2}-\left(R^2 \sinh^2\left(\frac{r}{R}\right)\right)^{-1}\right).
\end{eqnarray*}
In particular $\Phi$ is bounded above. 

If $M$ is a model space, i.e. $M$ is a manifold with a pole $p$ such that for every linear isometry $\phi: T_pM\to T_pM$ there exists an isometry $\Phi: M\to M$ such that $\phi(p)=p$ and $d\Phi_p=\phi$.
In the geodesic polar coordinates, $(r,\theta)\in (0, \infty)\times S^{n-1}$, the pull back metric in
$ (0, \infty)\times S^{n-1}$ can be written as $dr^2+f(r)^2 d\theta^2$. The function $f$ is $C^\infty$ and satisfies $f(0)=0$, $f'(0)=1$, $f(r)>0$ for $r>0$ and $f{''}(r)=-R(r)f(r)$ where $R(r)$ is the sectional curvature in a plane containing the radial direction $\partial_r$ at a point $x$ with $d(x,p)=r$.
Then  $\log J_p(x)=(n-1) \log \f {f(r)}r$. If $R(r)=R$, a constant, then
$f(r)=\f {\sinh (\sqrt R r)}{\sqrt R }$ and $\log f$ has bounded derivatives of all order. In general,
$$\begin{aligned} |\nabla \log J|&=(n-1) \left| (\log f)'(r)-\f 1r\right|, \qquad  \Delta r=(n-1) (\log f)'(r),\\
\Delta (\log J)&=(n-1)^2(\log f)'(r) \left ( (\log f)'(r)-\f {1}r \right) +(n-1) \left( (\log f)''(r)+\f {1} {r^2}\right),\\
\Phi&=\f 1 4 (n-1) \left[ \f{n-2} {r^2}    -\f {n-1} r(\log f)'(r) -(\log f)''(r) \right].
\end{aligned}$$

\subsection{Girsanov transform and zeroth order formula}\label{prelim}
Let $T$ be a positive number and $x_0, y_0 \in M$.  The semi-classical bridge process, between $x_0$ and $y_0$ in time $T$,  is the diffusion process starting at $x_0$ with generator  $\frac{1}{2}\Delta + \nabla \log k_{T-s}$ where
$k_t(x_0, y_0):=(2\pi t)^{-\f n 2}e^{-\f {d^2(x_0, y_0)}{2t}}J^{-\f 12}(x_0)$.
If $A$ is a vector field on $M$, we use the tilde sign over $A$ to denote its horizontal lift to the frame bundle. Following the notation in (\ref{horizontal}), we consider the equation on $OM$:
\begin{equation}\label{horizontal-2}
d\tilde u_t = \sum_{i=1}^n H_i(\tilde u_t) \circ dB^i_t+\h_{\tilde u_t} \left( {\nabla h}\right)dt+\h_{\tilde u_t} \left(\nabla \log k_{T-t}\right)dt, \quad t<T,
\end{equation}
where the initial value is a frame $u_0$ at $x_0\in M$ and $\h_u$ denotes the horizontal lift map from $T_\pi(u)M$ to $T_uOM$. 
Then  $\tilde x_t:=\pi(\tilde u_t)$ is a semi-classical bridge. 

Since $| \nabla r | = 1$ away from $y_0$, It\^{o}'s formula implies for $0\leq t<\term$ that
\begin{equation*}
r(\tilde{x}_t)-r(x_0)=
\beta_t +\int_0^t{1\over 2} \Delta r(\tilde{x}_s) ds-
\int_0^t  {r(\tilde{x}_s)  \over T-s}ds
 -{1\over 2} \int_0^t dr\left(  \nabla \log J(\tilde{x}_s) \right)ds
\end{equation*}
where $\beta_t$ is a standard one-dimensional Brownian motion. It is clear from this and the formula
\begin{equation}\label{laplacian-r}
\Delta r={n-1 \over r} +dr(\nabla \log J)
\end{equation}
that $r(\tilde{x}_t)$ is distributed as a Bessel bridge, starting at $r(x_0)$ and ending at $0$ at time $T$, from which it follows that $\lim_{t\uparrow T}\tilde{x}_t = y_0$, almost surely.

We follow a beautiful idea of Watling \cite{Watling1}
and use a modification of the construction by Elworthy-Truman to define the semi-classical Riemannian bridge
for a given vector field $Z$.   For $\gamma_x: [0,1]\to M$  the unique geodesic from $x$ to $y_0$, 
$$S(x)=\int_0^1 \<\dot \gamma_x(u), Z(\gamma_x(u))\>du,$$ 
and naturally $S(y_0)=0$. If $Z=\nabla f$, then $S(x)=f(y_0)-f(x)$ and $\nabla S+ \nabla f=0$.

\begin{lemma}
\label{lemma3.1}
If $\tilde x_t$ is a $\frac{1}{2}\Delta +Z+\nabla S+ \nabla (\log (k_{T-s}))$ diffusion (called the semi-classical Riemannian bridge), then  $d(y_0,\tilde x_t)$ is  the $n$-dimensional Bessel bridge process. 
\end{lemma}
\begin{proof}
Setting $r_t:=d(y_0, \tilde x_t)$, we have
$$r_t-r_0=
\beta_t +\f 12 \int_0^t \f {n-2}{r_s} ds-
\int_0^t  {r_s  \over T-s}ds+\int_0^t \< \nabla r, Z+\nabla S\>_{\tilde x_s} ds.$$
Then the last term in the identity vanishes.
Indeed,
$S(\gamma_x(t))=\int_t^1 \<\dot \gamma_x(u), Z(\gamma_x(u))\>du$ so
$\f d {dt}|_{t=0} S(\gamma_x(t))=-\<\dot \gamma_x(0), Z(x)\>$ on one hand,
 and $\f d {dt}|_{t=0} S(\gamma_x(t))=\<\nabla S(x), \dot \gamma_x(0)\>$  on the other hand. Finally $\nabla r(x)=\dot \gamma_x(0)$,
 and so $\<\nabla S+Z,\nabla r\>=0$.
\end{proof}

The following basic lemma will be used repeatedly, where $h\in C^2(M; \R)$, $Z$ is a $C^1$ vector field, and $t \in \left[0,\term\right)$.
We now consider a Brownian motion with drift $\nabla h+Z$. Although by the earlier discussion the symmetric drift can be absorbed into $Z$,
 it is convenient to treat $\nabla h$ differently from $Z$, the former is not involved in the definition of the semi-classical Riemannian bridge.
Both $\nabla h$ and $Z$ appear in the formulation of the Lemma (Lemma \ref{lemma-Girsanov}) below.
 Recall that $$\Phi=\f 12 J^{\f 12}\Delta (J^{-\f 12})=\f 12 |\nabla \log J|^2-\f 1 4 \Delta(\log J).$$
Let us define $$\Phi^h= \Phi -\f 1 2\Delta h - \f 1 2 |\nabla h|^2, \quad \Psi=
 \f 12 |\nabla S|^2+\f 12 \Delta S+\<Z, \nabla S-\nabla h\>.$$

\begin{lemma}\label{lemma-Girsanov}
Suppose that the $\frac{1}{2}\Delta^h +Z$ diffusion is complete.
 Then the probability distributions of the $\frac{1}{2}\Delta^h +Z$ diffusion and of the semi-classical Riemannian bridge, c.f. Lemma \ref{lemma3.1}, are mutually absolutely continuous on $\F_t$. Furthermore for any bounded measurable function $F$ on the path space $C([0,t];M)$, $\E F(u_\cdot)=\E [F(\tilde u_\cdot)M_t]$ where
 Radon-Nikodym derivative $M_t$, where $t<T$, is given by
 \begin{equation}
M_t:=\f {k_\term (x_0)e^{(S-h)(x_0)} } {k_{\term-\tm} (\tilde{x}_\tm) e^{(S-h)(\tilde x_t)}}\exp\left[  \int_0^\tm (\Phi^{h}+\Psi)(\tilde{x}_\stm) \,d\stm\right].
\end{equation}
In particular  $M_t$ is a martingale.
\end{lemma}

\begin{proof} 
For any $u_0\in OM$ with $\pi(u_0)=x_0$, let $u_s$ be the solution of the following equation with initial value $u_0$:
\begin{equation}
d {u}_s = \sum  H_i({u}_s)\circ dB^i_s +(\widetilde{\nabla h}+\widetilde Z) ({u}_s)\,ds.
\end{equation}
Then $u_t$ exists for all time and $ x_t=\pi(u_t)$ has generator $\f 12 \Delta^h+Z$. Let $\tilde u_t$ be the solution to
\begin{equation}
d \tilde{u}_s = \sum  H_i(\tilde{u}_s)\circ dB^i_s +(\widetilde{\nabla \log k_{T-s}}+\widetilde Z+\widetilde {\nabla S}) (\tilde{u}_s)\,ds,\quad  \tilde u_0=u_0
\end{equation}
Then $\tilde x_s=\pi(\tilde u_s)$ is a semi-classical Riemannian bridge, the Markov process with generator $\f 12 \Delta +Z+\nabla S+\nabla (\log k_{T-s})$. 
One simple, and yet crucial observation, is that
$$\frac{1}{2}\Delta^h +Z+\nabla \log (k_{T-s}e^{S-h})=\frac{1}{2}\Delta +Z+\nabla S+\nabla (\log k_{T-s}),$$
and so the generators of $\tilde x_t$  and $x_t$ differ by the  drift $\nabla \log (k_{T-s}e^{S-h})$. 
Since both stochastic processes are well defined before time $T$, they are equivalent on $\F_t$ for any $t<T$.

Let $\{e_i\}_{i=1}^n$ be an orthonormal basis of $\R^n$ and let $$\tilde B_t^i:=B_t^i +\int_0^t d \log( k_{T-s} e^{S-h})(\tilde u_s e_i)\, ds.$$ 
Then $\{\tilde B_t^i\}_{i=1}^n$ are independent Brownian motions on $(\Omega, \F, \F_t, {\mathbb Q})$ where, on each $\sigma$-algebra $\F_t$,  $\mathbb Q$ is defined to be the probability measure equivalent to $\mathbb P$ on each $\F_t$ with  $M_t=\f {d{\mathbb Q}}{d{\mathbb P}}|_{\F_t}$
given by
\begin{equation*}
M_t = \exp \left[ - \sum_{i=1}^m\int_0^\tm  \<\nabla \log (k_{\term-\stm} e^{S-h}) (\tilde x_\stm), \tilde  u_\stm e_i\> dB_\stm^i -{1\over 2} \int_0^\tm | \nabla \log (k_{\term-\stm}e^{S-h}) |_{(\tilde x_\stm)}^2 d\stm\right].
\end{equation*}
Now It\^{o}'s formula implies
\begin{equation*}
\begin{aligned}
&\log {(e^{S-h}k_{\term-\tm})(\tilde{x}_{\tm}) }= \text{ }\log {(e^{S-h} k_\term)(x_0)}
+\sum_{i=1}^m \int_0^{\tm}  \<\nabla \log ( k_{\term-\stm}e^{S-h})  (\tilde x_\stm), \tilde  u_\stm e_i\> dB_\stm^i \\
&+\int_0^{\tm} |\nabla \log (k_{\term-\stm}e^{S-h})(\tilde x_\stm)|^2\,d\stm+\int_0^{\tm} \left({\partial \over \partial s} +{1\over 2} \Delta^h +Z\right)\log ( k_{\term-\stm}e^{S-h})(\tilde x_\stm) \,d\stm,
\end{aligned}
\end{equation*}
using which we may eliminate the stochastic integral appearing in the formula for $M_t$. Indeed,  by the relation (\ref{laplacian-r}) we see that
\begin{equation}
\begin{aligned}
\f{\partial}{\partial s} \log k_{T-s} &=\f n  {2(T-s)}-\f {r^2} {2(T-s)^2},\\
|\nabla \log (k_{T-s}e^{- h})|^2&=\f {r^2}{(T-s)^2}
+\f 14|\nabla \log J|^2+\f {rd r(\nabla \log J)}{T-s}-2\< \nabla  h, \nabla \log (k_{T-s})\>+|\nabla  h|^2,\\
\Delta \log (k_{T-s}e^{- h}) &=-\f n {T-s} -\f {rd r (\nabla \log J)}{T-s}-\f 12 \Delta(\log J)-\Delta  h.
\end{aligned}\nonumber
\end{equation}
We have the following identity:
\begin{equation*}
\begin{aligned}
&\f 1 2|\nabla \log (k_{\term-\stm}e^{-h})|^2 + \left({\partial \over \partial s} +{1\over 2} \Delta^h +Z\right)\log( k_{T-s}e^{-h})\\
 &=\f 18 |\nabla \log J|^2+\f 1 2 |\nabla h|^2-\f 1 4 \Delta(\log J)-\f 1 2\Delta h-|\nabla h|^2+\<Z, \nabla \log (k_{T-s} e^{-h})\>\\
 &=\f 18 |\nabla \log J|^2-\f 1 4 \Delta(\log J)-\f 1 2\Delta h-\f 12 |\nabla h|^2+\<Z, \nabla (\log k_{T-s}e^{-h})\>,
\end{aligned}
\end{equation*}
from which we deduce the formula with non-vanishing  $S$ ,
\begin{equation*}
\begin{aligned}
&\f 1 2|\nabla \log (k_{\term-\stm}e^{-h+S})|^2 + \left({\partial \over \partial s} +{1\over 2} \Delta^h +Z\right)\log( k_{T-s}e^{-h+S})\\
 =&\f 12 |\nabla S|^2+\<\nabla S, \nabla \log (k_{T-s}e^{- h})\>+\f 12 \Delta^h S+\<Z, \nabla S\>
 +\<Z, \nabla \log (k_{T-s}e^{-h})\>\\
 &+\f 18 |\nabla \log J|^2-\f 1 4 \Delta(\log J)-\f 1 2\Delta h-\f 12 |\nabla h|^2.
\end{aligned}
\end{equation*}
Since $\nabla S+Z$ vanishes on radial directions, see the proof for Lemma \ref{lemma3.1}, the first 
line on the right hand side of the equation is
\begin{equation*}
\begin{aligned}
 &\f 12 |\nabla S|^2-\<\nabla S+Z, \nabla h\>+\f 12 \Delta S+\<\nabla h, \nabla S\>+\<Z, \nabla S\>\\
 &= \f 12 |\nabla S|^2+\f 12 \Delta S+\<Z, \nabla S-\nabla h\>.
\end{aligned}
\end{equation*}
Finally we obtain

\begin{equation*}
\begin{aligned}
&\f 1 2|\nabla \log (k_{\term-\stm}e^{-h+S})|^2 + \left({\partial \over \partial s} +{1\over 2} \Delta^h +Z\right)\log( k_{T-s}e^{-h+S})\\\
 &= \f 12 |\nabla S|^2+\f 12 \Delta S+\<Z, \nabla S-\nabla h\>+\Phi-\f 1 2\Delta h-\f 12 |\nabla h|^2
\end{aligned}
\end{equation*}
and consequently,
$$
\begin{aligned}
\log {(e^{S-h}k_{\term-\tm})(\tilde{x}_{\tm}) }= \text{ }&\log {(e^{S-h} k_\term)(x_0)}
+\sum_{i=1}^m \int_0^{\tm}  \<\nabla \log ( k_{\term-\stm}e^{S-h})  (\tilde x_\stm), \tilde  u_\stm e_i\> dB_\stm^i \\
&+\f 12 \int_0^{\tm} |\nabla \log (k_{\term-\stm}e^{S-h})(\tilde x_\stm)|^2\,d\stm
+\int_0^{\tm} (\Phi^h+\Psi)(\tilde x_\stm) \,d\stm,
\end{aligned}$$
and $$M_t=\f {(e^{S-h} k_\term)(x_0)}{(e^{S-h}k_{\term-\tm})(\tilde{x}_{\tm}) }\exp\left(\int_0^{\tm} (\Phi^h+\Psi)(\tilde x_\stm) \,d\stm\right),$$
completing the proof.
\end{proof}

Elworthy and Truman's proof of the following theorem, for the case $h=0$,  was inspired by semiclassical mechanics. They used a semiclassical bridge which arrives at $y_0$ at time $\term+\lambda$ and took the limit as $\lambda \downarrow 0$. We give a slightly modified proof, generalising their result for $\Delta^h$, the method of which will later be used to derive a derivative formula. The following generalises formulas in \cite{Elworthy-Truman-81,Watling1}.

\begin{theorem}\label{eltruform}
Let $V\in C^{0,\alpha}\cap L_\infty$ and suppose that the Brownian motion with drift $\nabla h+Z$ does not explode.  Suppose that  $\beta^h_t$ converges in $L^1$ as $t\to T$, which holds if $\Phi^h+\Psi-V$ is bounded above.  Then
\begin{equation}\label{elaform-1} 
p^{h,Z,V}_T(x_0,y_0) =e^{(h-S)(y_0)-(h-S)(x_0)}k_T(x_0, y_0)\E \left[ \exp\left(\int_0^T (\Phi^h+\Psi-V)(\tilde{x}_s)ds\right) \right].
\end{equation}
\end{theorem}

\begin{proof}
Let $\phi$ be a smooth function with compact support and such that $\phi(y_0)=1$. Denote by $E_t$ the standard Gaussian kernel in the tangent space $T_{y_0} M$. Then, by a change of variables, we see that
\begin{equation}
\begin{aligned}
&\text{ }\lim_{t\uparrow T}\lim_{r\uparrow t} P^{h,Z,V}_t (\phi k_{T-r})(x_0)= \lim_{t\uparrow T} \lim_{r\uparrow t} \int_M p^{h,Z,V}_t(x_0,y)\phi(y)k_{T-r}(y)\,dy\\
=& \text{ } \lim_{t\uparrow T}\lim_{r\uparrow t} \int_{T_\pole M} (p^{h,Z,V}_t(x_0,\cdot)\cdot\phi \cdot J^{1/2})(\exp_\pole v)E_{T-r}(v)\,dv \\
=& \text{ }(p^{h,Z,V}_T(x_0,\cdot)\cdot \phi\cdot J^{1/2})(\exp_\pole 0_\pole)
=\text{ } p^{h,Z,V}_T (x_0,\pole)
\end{aligned}\nonumber
\end{equation}
where $0_\pole$ denotes the origin of the tangent space $T_\pole M$. We have used the fact that $p^{h,Z,V}_t(x_0,\cdot)\cdot \phi\cdot J^{1/2}$ has compact support with $J(\exp_\pole 0_\pole)=J(\pole)=1$.

Suppose that $f \in L_\infty$ and set \begin{equation}
\beta^{h,Z}_t=\exp\left(  \int_0^t (\Phi^h+\Psi-V)(\tilde{x}_s)ds\right).
\end{equation}
 Then Lemma \ref{lemma-Girsanov} implies that for $0 \leq t < T$,
\begin{equation}\label{girsap}
\begin{aligned}
\int_M f(y) p^{h,Z,V}_t(x_0,y)\,dy =\E [\V_t f(x_t)]=
k_T(x_0,y_0) e^{(S-h)(x_0)}\mathbb{E}\left[\f { f(\tilde x_t)\beta^{h,Z}_t}{k_{T-t}(\tilde x_t)e^{(S-h)(\tilde x_t)}} \right].
\end{aligned}
\end{equation}

Thus, by taking $f=\phi \cdot k_{T-r}$ in equation (\ref{girsap}) for $r<t$, we observe, since $\tilde x_t$ converges to $y_0$ a.s., that
\small
\begin{equation}
\begin{aligned}
p^{h,Z,V}_T(x_0,\pole) =& e^{(S-h)(x_0)}k_T(x_0, y_0) \lim_{t\uparrow T} \lim_{r\uparrow t}
\E \left[ \phi(\tilde x_t) \f { k_{T-r}(\tilde x_t) }{k_{T-t}(\tilde x_t)e^{(S-h)(\tilde x_t)}} \exp\left(\int_0^t (\Phi^h+\Psi-V)(\tilde{x}_s)ds\right)\right]\\
=\text{ }&e^{(S-h)(x_0)-(S-h)(y_0)}k_T(x_0, y_0)\E\left[ \exp\left(\int_0^T (\Phi^h+\Psi-V)(\tilde{x}_s)ds\right)\right],
\end{aligned}\nonumber
\end{equation}
\normalsize
Since $\phi$ has compact support the limit $r\to t$ is trivial to justify,   those terms singular at $t=T$ cancel out after taking the limit. The second follows from
the assumption. The proof is complete. 
\end{proof}

At this point we compare formula (\ref{elaform-1}), valid for all time $t\le T$ with  Varadhan's asymptotic relation \cite{Varadhan} and the asymptotic expansion of Minakshisundaram and  Pleijel  \cite{Minakshisundaram-Pleijel} for small time. The first states that  $\lim_{t\downarrow 0} 2t\log p_t \rightarrow -d^2$, uniformly on compact sets, which was proved in  \cite{Varadhan} for operators of the form $\sum_{i,j} a_{i,j}{\partial ^2} {\partial x_i \partial x_j}$ in $\R^n$ and extends to heat kernels on complete Riemannian manifolds by Molchanov \cite{Molchanov-75},  the completeness  condition was added in ( Azencott et al  \cite{Azencott-small-time}) and is needed.
 The latter is an asymptotic expansion, states that there are smooth functions $H_i$ defined on $M\times M\setminus \cut(M)$ such that
$$p_t(x,y)\sim  (2\pi t)^{-\frac{n}{2}} e^{-\frac{d^2(x,y)}{2t}}\sum_{i=0}^\infty H_i(x,y)t^i$$
with $H_0(x,y)={J_x^{-\frac{1}{2}}(y)}$, as $t\to 0$. Both converge uniformly on compact subsets of  $M\times M\setminus \cut(M)$, where $\cut(M)$ denotes the cut locus of $M$. 
See also \cite{Ariyoshi-Hino}. 

Our results  are global and  are valid only for manifolds with a pole. We hope that a suitable local result 
hold on a more  general manifold, using the cut-off procedure in \cite{Chen-Li-Wu-cut-off}.

\subsection{First order formulas}
In this section we take the drift field $Z$ to be zero.
Set ${D\over dt}\tilde W_t = \tilde u_t \frac{d}{dt} \tilde u_t^{-1} \tilde W_t$ where $\tilde u_t$ is the process on the orthonormal frame bundle, solving  the equation (\ref{horizontal-2}). We denote by $\tilde{W}_t$  the solutions to the following equation along the semi-classical Brownian bridge $\tilde x_t$ where $\tilde x_t=\pi(\tilde u_t)$:
\begin{equation}\label{damped2}
{D \over dt }\tilde{W}_t=-{1\over 2}  {\Ric}_{\tilde x_t}^{\#}(\tilde{W}_t)+\Hess(h)({\tilde W_t}), \quad \tilde W_0 = {\id}_{T_{x_0}M}.
\end{equation}

\begin{theorem}\label{thm:fof}
Suppose that $y_0$ is a pole, that $\rho^h$ is bounded from below, and that $\Phi^h$ is bounded above. 
Let $\tilde{W}$ be the solution to the damped parallel translation equation (\ref{damped2}). For $r\in [0, T)$ and $t\in [0,T]$  let
\begin{equation*}
d\tilde{B}_r = dB_r + \tilde{u}_r^{-1}\nabla (\log (k_{T-r}e^{-h}))(\tilde{x}_r)\,dr, \quad  
\beta_t^h=\exp\left(  \int_0^t (\Phi^h-V)(\tilde{x}_s)ds\right).\end{equation*}
\begin{enumerate}
\item[(a)]
Suppose that  $V \in  BC^1$. Then
\begin{equation*}\begin{aligned}
&T\, e^{-(h(y_0)-h(x_0))} \f{(dp^{h,V}_T(\cdot,y_0))_{x_0} }{k_T(x_0, y_0)}\\
&=\E \left[ \beta_T^h \left(\int_0^T \< \tilde{W}_r(\cdot),\tilde{u}_r d\tilde{B}_r\> -\int_0^T (T-r) dV(\tilde{W}_r(v)) dr\right)\right]. \end{aligned}
\end{equation*}
\item [(b)] Suppose that $V$ is H\"{o}lder continuous and bounded, then
\begin{equs}
&e^{-(h(y_0)-h(x_0))}\f{d {p^{h,V}_T} (\cdot,y_0)_{x_0}}{k_T(x_0, y_0)}=
 \frac{1}{T} \E \left[ \beta_T^h\int_0^T \< \tilde{W}_s(\cdot),\tilde{u}_s d\tilde B_s\rangle\right]\\
 &\hskip 45pt+ { \E \left[ \beta_T^h \int_0^T V (\tilde x_{T-s})  \;e^{-\int_{T-s}^T V(\tilde x_u)du }
\f  {1}{T-s}\int_0^{T-s} \< \tilde W_r(\cdot),\tilde u_r d\tilde B_r\> \right]}.
\end{equs}
\end{enumerate}
\end{theorem}

\begin{proof}
For all $0< t \leq T$ and $v\in T_{x_0}M$ we have, by Theorem \ref{thm:maintheorem} and Fubini's theorem, that
\begin{equation}
\begin{aligned}
dP^{h,V}_tf(v) &=\text{ }\f 1 t\E \left[ \V_t f(x_t) \int_0^t \< W_s(v),u_s dB_s\>-(t-s)dV(W_s(v)) \,ds\right]\\
&=\text{ } \f 1t\E \left[ \V_t f(x_t) \int_0^t \< W_s(v),u_s dB_s-(t-s)\nabla Vds\>\right].
\end{aligned}\nonumber
\end{equation}
Therefore, it follows that, for all $0< t <T$,
\begin{equation}
\begin{aligned}
dP^{h,V}_t(\phi k_{T-t})(v) &=\text{ }\frac{1}{t}\E \left[ \V_t (\phi k_{T-t})(x_t) \int_0^t \< W_r(v),u_r dB_r-(t-r)\nabla Vdr\>\right],
\end{aligned}\nonumber
\end{equation}
 where $\phi$ is a smooth function with compact support and $\phi(y_0)=1$, as in the proof of Theorem \ref{eltruform}. By Lemma \ref{lemma-Girsanov} this yields
\begin{equation}
\begin{aligned}
&dP^{h,V}_t(\phi k_{T-t})(v)\\
&= \frac{1}{t} e^{-h(x_0)}k_T(x_0,y_0)\E \left[ \phi(\tilde x_t)\beta_t^h 
e^{h(\tilde x_t)}\int_0^t \< \tilde{W}_r(v),\tilde{u}_r d\tilde{B}_r-(t-r)\nabla Vdr\>\right].
\end{aligned}\nonumber
\end{equation}
We now take limits. For the left-hand side of the previous equation we see that
\begin{equation}
\begin{aligned}
&\text{ }\lim_{t\uparrow T}dP^{h,V}_t(\phi k_{T-t})(v)
=\lim_{t\uparrow T}d\left(\int_M p^{h,V}_t(\cdot,y)\phi(y)k_{T-t}(y)\,dy\right)(v)\\
=&\text{ } \lim_{t\uparrow T}\int_{M} dp^{h,V}_t(\cdot,y)(v)\phi(y)k_{T-t}(y)\,dy
=dp^{h,V}_T(\cdot,y)(v)
\end{aligned}\nonumber
\end{equation}
where the final equality follows from the compactness of the support of $\phi$ and the fact that $J(y_0)=1$. 
For the right-hand side, we use $\lim_{t\to T} \tilde x_t=y_0$, 
apply the dominated convergence theorem, using upper and lower bounds on $\Phi^h$ and $\Ric-2 \Hess h$, respectively, obtaining an upper bound on $|\tilde W_r|$ and the result as claimed.
\end{proof}

 In the paper \cite[Theorem 6]{CLY} by  Cheng, Li and Yau, an estimate of the following form
 is given for Riemannian manifold of bounded curvature,
$$\nabla p_t(x_0, y_0)\le C \delta^{-\alpha(n)}(x_0)t^{-(n+\f 12)}e^{- \f{\alpha(n)d^2(x_0, y_0)}{t}}.$$
where $C$ is a constant depending on $n$, on the  bound on the sectional curvature and on $ t$, 
$ \delta(x_0)$ is a constant depending on the injectivity radius of $x_0$, and $\alpha(n)$ is a universal constant. Let $\underline \Ric_x$ denote the lower bound of the Ricci curvature at $x$.
We have the following corresponding estimates.

\begin{corollary}\label{pole-estimate}
Suppose that $y_0$ is a pole,  $\Phi^h-V$ is bounded above,   $\underline\Ric-2 \Hess (h)\ge K$, and $\nabla h$, and $\nabla \log J$ are bounded. Suppose that $V \in  BC^1$.
Then for an explicit constant $C$ depending only on $n$, $K$, $T$, and  $|\nabla \log J|_\infty$ (and locally bounded in $T$), we have
$$\f{|\nabla p^{h,V}_T(\cdot,y_0)|_{x_0}}{k_T(x_0,y_0)} \le Ce^{h(y_0)-h(x_0)} |\beta^h_T|_\infty \left(\f  {d(x_0, y_0)}T+ |\nabla h|_\infty+1+ \f 1 {\sqrt T}+T|dV|_\infty\right) .$$
\end{corollary}
\begin{proof}
Since $\tilde W$ is defined by (\ref{damped2}) and $|\tilde W_t(v)|\le e^{-\f 12 Kt}|v|$, $$\begin{aligned} 
&\left|  \int_0^T \< \tilde{W}_r(v),\tilde{u}_r d\tilde{B}_r\> -\int_0^T (t-r) dV(\tilde{W}_r(v)) dr \right| \\
&= \left|  \int_0^T \< \tilde{W}_r(v),\tilde{u}_r d{B}_r\> + \int_0^T \left\<\tilde W_r(v), \nabla \log (e^{-h} k_{T-r})(\tilde{x}_r)\right \>\,dr
-\int_0^T (t-r) dV(\tilde{W}_r(v)) dr\>\right| \\
&\le \left|  \int_0^T \< \tilde{W}_r(v),\tilde{u}_r d{B}_r\>\right| +|v|\;  \int_0^T e^{-Kr/2} \left|\nabla \log k_{T-r}(\tilde x_r)-\nabla h)(\tilde{x}_r)\right|dr\\
&\quad+|v|\; |dV|_\infty \int_0^T e^{-Kr/2} (t-r) dr.
 \end{aligned}$$
Since  $\E  \left|  \int_0^T \< \tilde{W}_r(v),\tilde{u}_r d{B}_r\>\right|\le \sqrt{\E \int_0^T |\tilde W_r(v)|^2dr}
 \le \left( \int_0^T e^{-Kr}|v|^2dr\right)^{\f 12}$, we apply Theorem \ref{thm:fof} to see that
$$\begin{aligned} 
&\left|e^{h(x_0)-h(y_0)} \f{\nabla p^{h,V}_T(\cdot,y_0)(v)}{k_T(x_0, y_0)} \right|\\
\le&\f 1 T|v|\;  |\beta^h_T|_\infty \left(\int_0^T e^{-Kr}dr \right)^{\f 12}+\f 1 T|v|\;  |\beta^h_T|_\infty |dV|_\infty \int_0^T e^{-Kr/2} (T-r) dr\\
+&\f 1 T|v|\;  |\beta^h_T |_\infty\E \left[  \int_0^T e^{-Kr/2} (|\nabla h|_\infty+ |\nabla \log k_{T-r}(\tilde{x}_r)|)dr\right].
 \end{aligned}$$
The sum of the first two terms on the right hand side are nicely bounded by 
$$|v|\; |\beta^h_T|_\infty  \f 1 T(\sqrt{C_1(K,T) }+|dV|_\infty \,C_2(K, T))$$ where $C_1$ and $C_2$ are the obvious integrals of order $T$: $C_1(K,T)=\frac {1-e^{-KT} } K$ and $C_2= \int_0^T e^{-Kr/2} (T-r) dr$. They are  of order $T$ for small time.
Since $\nabla \log k_t(\cdot, y_0)=-\f {\nabla d^2}{2t} -\f 12 \nabla \log J$, the last term can be estimated using
the Euclidean bridge. Then
$$\begin{aligned} 
&\E \left[  \int_0^T e^{-Kr/2} |\nabla \log k_{T-r}(\tilde{x}_r)|\,dr\right]\\
&\le \f 12     \int_0^T e^{-Kr/2}\, |\nabla \log J|_\infty\, dr
+\E \left[   \int_0^T e^{-Kr/2} \f { d(\tilde x_r, y_0)}{T-r} \,dr\right]\\
&\le \f 12    |\nabla \log J|_\infty \int_0^T e^{-Kr/2} \, dr
+ \int_0^T  e^{-Kr/2} \f {\sqrt{ \E d^2(\tilde x_r, y_0)}}{(T-r)}\, dr.
 \end{aligned}$$
 Since $r_t=d(\tilde x_t, y_0)$ is the $n$-dimensional Bessel bridge, 
 $$ \E d^2(\tilde{x}_r , y_0)=\left(\f {T-r}T\right)^2 d^2(x_0, y_0)+\f {nr} T (T-r),$$
 and consequently,
 $$ \int_0^t  e^{-Kr/2} \f {\sqrt{ \E d^2((\tilde{x}_r) , y_0)}}{(T-r)} dr \le  \int_0^t  e^{-Kr/2} \left(    \f 1 T d(x_0, y_0) +\sqrt{ \f {nr} T }\f 1 {\sqrt{T-r}}\right)  dr,$$
 which is bounded by $ C (d(x_0, y_0)+\sqrt T)$.
 This completes the proof.
\end{proof}

Theorem \ref{thm:fof} together with the formulas for the fundamental solutions, we obtain formulas for the logarithmic derivatives of the fundamental solutions. 

\begin{proposition}
\label{gradient-kernel}
 Assume that $y_0$ is a pole for $M$,  $\underline{\Ric}-2\Hess (h)$ is bounded below
 and $\Phi^h$ is bounded above. Set $Z_T ^h= \frac{\beta^h_T}{\E(\beta^h_T)}=\frac{\exp\left(  \int_0^t (\Phi^h-V)(\tilde{x}_s)ds\right)}
{\E \exp\left(  \int_0^t (\Phi^h-V)(\tilde{x}_s)ds\right)}$.
 \begin{enumerate}
\item [(a)] If $V$ is H\"{o}lder continuous and bounded, then
 \begin{equs}
d {\log p^{h,V}_T} (\cdot,y_0)_{x_0}=
 &\frac{1}{T} \E \left[ Z_T^h\int_0^T \< \tilde{W}_s(\cdot),\tilde{u}_s d\tilde B_s\rangle\right]
 \\
&+ { \E \left[ Z_T^h \int_0^T V (\tilde x_{T-s})  \;e^{-\int_{T-s}^T V(\tilde x_u)du }
\f  {1}{T-s}\int_0^{T-s} \< \tilde W_r(\cdot),\tilde u_r d\tilde B_r\> \right]}.
\end{equs}
 \item[(b)]  If $V\in BC^1$, then
\begin{equation*}
d\log p^{h,V}_T(\cdot,y_0)(v)
= \frac{1}{T}\E \left[ Z_T^h \int_0^T \< \tilde{W}_r(v),\tilde{u}_r d\tilde{B}_r-(T-r)\nabla Vdr\>\right].
\end{equation*}
 \end{enumerate}
\end{proposition}
These are immediate consequences of  Theorems \ref{Feynman-Kac-kernel-1}, \ref{eltruform} and \ref{thm:fof}.

Finally we conclude the paper with  an estimate for the gradient
of the logarithmic of the Feynman-Kac kernel. \begin{corollary}
Assume that $y_0$ is a pole, $V\in BC^1$, $\rho^h$ is bounded from below, $\Phi^h$ is bounded above,
$|\nabla h|$ and $|\nabla \log J|$ are bounded. Then
$$\left|\nabla \log p^{h,V}_T(\cdot,y_0)\right|_{x_0}\le C |Z_T^h|_\infty \left( \f {d(x_0, y_0)}T+1+|\nabla h|_\infty+\f 1{\sqrt T}+T\,|dV|_\infty \right),$$
where $C$ is a constant.
\end{corollary}
The proof will be entirely analogous to that for Corollary \ref{pole-estimate}. 
We simply use part (b) of Proposition \ref{gradient-kernel} and observe that the formula for 
$\nabla \log p^{h,V}_T(\cdot,y_0)$ is similar to that for 
$e^{h(x_0)-h(y_0)}\f{\nabla  p^{h,V}_T(\cdot,y_0)}{k_t(\cdot, y_0)}$ with $\beta_t^h$  replaced by its average $Z_T^h$. Since $Z_T^h$ is bounded by our assumption, the remaining estimate is the same as that used for the proof of Corollary \ref{pole-estimate}. 
\\

\vspace{2mm}
\noindent
{\bf Acknowledgment.} 
 It is our pleasure to thank Professor K. D. Elworthy  for helpful conversations.
\def\cprime{$'$} \def\cprime{$'$} \def\cprime{$'$} \def\cprime{$'$}

\end{document}